\documentclass[12pt]{amsart}
\usepackage{amssymb,amsmath,amsthm,amssymb}
\usepackage{mathtools}
\usepackage{graphicx}
\usepackage{comment}
\usepackage{hyperref}
\usepackage{enumerate}
\usepackage{color}
\usepackage{epstopdf}
\renewcommand{\epsilon}{\varepsilon}

\def\R{\mathbb{R}}

\def\e{\epsilon}
\def\t{\tau}

\def\th{\theta}

\newcommand{\pd}{\partial}
\newcommand{\cd}{\nabla}

\def\ba #1\ea {\begin{align} #1\end{align}}
\def\bann #1\eann {\begin{align*} #1\end{align*}}
\def\ben #1\een {\begin{enumerate} #1\end{enumerate}}
\def\bi #1\ei {\begin{itemize}\renewcommand\labelitemi{--} #1\end{itemize}}

\newcommand{\inner}[2]{\left\langle#1,#2\right\rangle} 

\newtheorem{theorem}{Theorem}[section]
\newtheorem*{theorem*}{Theorem}
\newtheorem{lemma}[theorem]{Lemma}
\newtheorem{proposition}[theorem]{Proposition}
\newtheorem{remark}[theorem]{Remark}
\newtheorem{corollary}[theorem]{Corollary}

\newtheorem*{conjecture*}{Conjecture}

\newtheorem*{claim*}{Claim}

\newtheoremstyle{TheoremNum}
        {\topsep}{\topsep}              
        {\itshape}                      
        {}                              
        {\bfseries}                     
        {.}                             
        { }                             
        {\thmname{#1}\thmnote{ \bfseries #3}}
    \theoremstyle{TheoremNum}

\author{Theodora Bourni}
\author{Mat Langford}
\address{Department of Mathematics, University of Tennessee Knoxville, Knoxville TN, 37996-1320}
\email{tbourni@utk.edu, mlangford@utk.edu}
\address{School of Mathematical and Physical Sciences, The University of Newcastle, Newcastle, NSW, Australia, 2308}
\email{mathew.langford@newcastle.edu.au}

\date{\today}

\title[Convex ancient free boundary MCF in the ball.]{Classification of convex ancient free boundary mean curvature flows in the ball.}

\begin{document}

\begin{abstract}
We prove that there exists, in every dimension, a unique (modulo rotations about the origin and time translations) convex ancient mean curvature flow in the ball with free boundary on the sphere. This extends the main result of \cite{FBCSF} to general dimensions.
\end{abstract}

\maketitle

\setcounter{tocdepth}{1}
\tableofcontents

\section{Introduction}

Mean curvature flow is the gradient flow of area for regular submanifolds. It models evolutionary processes governed by surface tension, such as grain boundaries in annealing metals \cite{Mullins,vonNeumann}.

The systematic study of mean curvature flow was initiated by Brakke \cite{Brakke} (from the point of view of geometric measure theory) and later taken up by Huisken \cite{Hu84}, who proved that closed convex hypersurfaces remain convex and shrink to ``round'' points in finite time. Different proofs of Huisken's theorem were discovered later by others \cite{An94,An12,HamiltonPinched}.

Ancient solutions (that is, solutions defined on backwards-infinite time-intervals) are important from an analytical standpoint as they model singularity formation \cite{MR1375255}. They also arise in quantum field theory, where they model the ultraviolet regime in certain Dirichlet sigma models \cite{Bakas2}. They have generated a great deal of interest from a purely geometric standpoint due to their symmetry and rigidity properties (see, for example, \cite{BLTatomic} and the references therein). 

The natural Neumann boundary value problem for mean curvature flow, called the \emph{free boundary problem}, asks for a family of hypersurfaces whose boundary lies on (but is free to move on) a fixed barrier hypersurface which is met by the solution hypersurface orthogonally. Study of the free boundary problem was initiated by Huisken \cite{HuiskenNonparametric} and further developed by Stahl \cite{Stahl2,Stahl1} and others \cite{MR3479560,MR2718258,MR3200337,MR3150205}. In particular, Stahl proved that convex hypersurfaces with free boundary on a totally umbilic barrier remain convex and shrink to a ``round'' point on the barrier hypersurface. Hirsch and Li \cite{HirschLi} proved that the same conclusion holds for ``sufficiently convex'' barriers in $\R^3$.

The analysis of ancient solutions to free boundary mean curvature flow remains in its infancy. We recently classified the convex ancient solutions to curve shortening flow in the disc \cite{FBCSF} but, to our knowledge, the only other examples previously known seem to be those inherited from closed or complete examples (one may restrict the shrinking sphere, for example, to a halfspace).

We provide here a classification of convex\footnote{A free boundary hypersurface of the open ball $B^{n+1}$ is \emph{convex} if it bounds a convex region in $B^2$ and \emph{locally uniformly convex} if it is of class $C^2$ and its second fundamental form is positive definite.} ancient free boundary mean curvature flows in the ball in all dimensions.

\begin{theorem}\label{thm:classification}
Modulo rotation about the origin and translation in time, there exists, for each $n\ge 1$, exactly one convex, locally uniformly convex ancient solution to free boundary mean curvature flow in the $(n+1)$-ball $B^{n+1}$. It converges to the point $e_{n+1}$ as $t\to 0$ and smoothly to the horizontal bisector $B^{n}\times\{0\}$ as $t\to-\infty$. It is invariant under rotations about the $e_{n+1}$-axis. As a graph over $B^{n}$, it satisfies
\[
\mathrm{e}^{\mu_0t}u(x,t)\to A\phi_0(x)\;\;\text{uniformly in $x$ as}\;\; t\to-\infty
\]
for some $A>0$, where 
$\mu_0<0$ is the lowest eigenvalue and $\phi_0$ the corresponding ground state of the ``critical-Robin'' Laplacian on $B^n$.
\end{theorem}
Theorem \ref{thm:classification} is a consequence of Propositions \ref{prop:existence}, \ref{prop:linear analysis}, and \ref{prop:uniqueness} proved below. Note that it is actually a classification of all convex ancient solutions, since the strong maximum principle and the Hopf boundary point lemma imply that any convex solution to the flow is either a stationary hyperball (and hence a bisector of the $(n+1)$-ball by the free boundary condition) or is locally uniformly convex at interior times.

Our proof is strongly motivated by the one-dimensional case \cite{FBCSF}. The main differences involve the use of the critical-Robin ground state in the construction of the solution and of a Liouville theorem for the critical-Robin heat equation (which we prove in \S\ref{sec:logNeumann}) in proving its uniqueness. 



\subsection*{Acknowledgements}
ML would like to thank Julie Clutterbuck for useful discussions about the Robin Laplacian, and the MATRIX Institute for its hospitality during the \emph{Isoperimetric Inequalities in Geometric Partial Differential Equations} workshop where these discussions took place.

TB was supported through grant 707699 of the Simons Foundation and grant DMS 2105026 of the National Science Foundation. ML was supported through an Australian Research Council DECRA fellowship (grant DE200101834).



\section{The critical-Robin heat equation}\label{sec:logNeumann}

We will need\footnote{Most (if not all) of the results of this section appear to be well-known (they are simple consequences of, for example, the results of \cite[\S 5]{Robin}) but they are easy enough to obtain directly, so we include the details here.} a Liouville theorem for ancient solutions to the ``critical-Robin heat equation'' on the unit $n$-ball $B^n$; that is, solutions 
$u$ to the problem
\begin{equation}\label{eq:heat equation on the interval}
\left\{\begin{aligned}
(\pd_t-\Delta)u={}&0\;\;\text{in}\;\; B{}^n\times(-\infty,\infty)\\
\cd u\cdot x={}&u\;\;\text{on}\;\; \pd B^n\,.
\end{aligned}\right.
\end{equation}
Separation of variables leads us to consider eigenfunctions of the ``critical-Robin Laplacian'' on $B^n$; that is, solutions 
$u$ to the problem
\begin{equation}\label{eq:eigenfunction}
\left\{\begin{aligned}
-\Delta\phi={}&\mu\phi\;\;\text{in}\;\; B^n\\
\cd\phi\cdot x={}&\phi\;\;\text{on}\;\;\pd B^n\,.
\end{aligned}\right.
\end{equation}
Observe that the lowest eigenvalue $\mu_0$ is variationally characterized by
\[
\mu_0
=\inf_{
u\in H^1(B^n)}\frac{B(u,u)}{\vert u\vert^2_{L^2(B^n)}}\,,
\]
where\footnote{The second integral is interpreted according to the trace theorem.}
\[
B(u,v)\doteqdot \int_{B^n}\cd u\cdot\cd v-\int_{\pd B^n}uv\,.
\]
We note that standard methods (cf. \cite[\S6.5.1]{EvansBook}) ensure that $\mu_0$ is finite and simple, the spectrum forms a sequence $-\infty<\mu_0<\mu_1\le\dots\le \mu_k\le\dots$ (with each eigenvalue appearing according to its multiplicity and $\mu_k\to\infty$), and $L^2(B^n)$ admits an orthonormal basis consisting of corresponding eigenfunctions (which are necessarily smooth).

Separation of variables leads us to consider solutions of the form $u(r,z)=\phi(r)\Phi(z)$, where $\phi$ satisfies
\begin{equation}\label{eq:radial eigenfunction higher eigenvalues}
-\left(\phi_{rr}+\frac{n-1}{r}\phi_r-\frac{\ell(\ell+n-2)}{r^2}\phi\right)=\mu\phi\;\;\text{in}\;\; (0,1)
\end{equation}
and $\Phi$ satisfies
\[
-\Delta_{S^{n-1}}\Phi=\ell(\ell+n-2)\Phi
\]
for some non-negative integer $\ell$.


\begin{lemma}\label{lem:negative eigenspace}
The negative eigenspace of \eqref{eq:eigenfunction} is one-dimensional.
\end{lemma}
\begin{proof}
Assuming $\mu<0$, consider the radial problem \eqref{eq:radial eigenfunction higher eigenvalues}. 
We make the change of variable $r\mapsto \rho\doteqdot \lambda r$, where $\lambda\doteqdot \sqrt{-\mu}$, which results in the problem
\begin{equation}\label{eq:negative eigenspace change of variable}
\left\{\begin{aligned}
\phi''+\frac{n-1}{\rho}\phi'-\frac{\ell(\ell+n-2)}{\rho^2}\phi={}&\phi\;\;\text{in}\;\; (0,\lambda)\\
\lambda\phi'(\lambda)={}&\phi(\lambda)\,.
\end{aligned}\right.
\end{equation}
%
Consider a formal Taylor series solution
\[
\phi(\rho)=\sum_{j=0}^\infty a_j\rho^j
\]
to \eqref{eq:negative eigenspace change of variable}. Observe that
\[
\phi'(\rho)=\sum_{j=1}^\infty ja_j\rho^{j-1}\;\;\text{and}\;\;\phi''(\rho)=\sum_{j=2}^\infty j(j-1)a_j\rho^{j-2}
\]
and hence $\phi$ satisfies \eqref{eq:negative eigenspace change of variable} if and only if
\bann
0={}&\phi''+\frac{n-1}{\rho}\phi'-\frac{\ell(\ell+n-2)}{\rho^2}\phi-\phi\\
={}&\frac{(n-1)a_1-\ell(\ell+n-2)a_1}{\rho}-\frac{\ell(\ell+n-2)a_0}{\rho^2}\\
{}&+\sum_{j=2}^\infty\big(j(n+j-2)a_j-\ell(n+\ell-2)a_j-a_{j-2}\big)\rho^{j-2}\,.
\eann
Since $\phi\in L^2(0,1)$, the first two terms must vanish. Lest $\phi\equiv 0$, we conclude that either 1. $a_0=0$ and $\ell=1$ (and hence $\Phi$ is linear), or 2. $a_1=0$ and $\ell=0$ (and hence $\Phi$ is constant).

In the first case,
\[
(j-1)(j+n-1)a_j=a_{j-2}\;\;\text{for}\;\; j\ge 2\,,
\]
so that
\[
\phi(\rho)=\sum_{j=0}^\infty b_j\rho^{2j+1}\,,\;\; b_j\doteqdot a_{2j+1}\,,
\]
where the coefficients are given by
\bann
b_j={}&b_0\prod_{\ell=1}^j\frac{1}{2\ell(n+2\ell)}\,.
\eann
We may assume that $b_0=1$. Since
\[
\phi'(\rho)=1+\sum_{j=1}^\infty (2j+1)b_j\rho^{2j}\,,
\]
the boundary condition yields the equation
\[
\sum_{j=1}^\infty 2jb_j\lambda^{2j}=0
\]
for the frequency $\lambda$, which is impossible since $\lambda>0$.

In the second case,
\[
j(n+j-2)a_j=a_{j-2}\;\;\text{for}\;\; j\ge 2\,,
\]
so that
\[
\phi(\rho)=\sum_{j=0}^\infty b_j\rho^{2j}\,,\;\; b_j\doteqdot a_{2j}\,,
\]
where the coefficients are given by
\bann
b_j={}&b_0\prod_{\ell=1}^j\frac{1}{2\ell(n+2(\ell-1))}\,.
\eann
We may assume that $b_0=1$. 

Since
\[
\phi'(\rho)=\sum_{j=1}^\infty 2jb_j\rho^{2j-1}\,,
\]
the boundary condition then yields the equation
\[
\sum_{j=1}^\infty(2j-1)b_j\lambda^{2j}=1\,
\]
for the frequency $\lambda$. Since the coefficients are all positive, the expression on the left is monotone in $\lambda$. Since it goes to zero when $\lambda\to 0$ and to infinity when $\lambda\to \infty$, there must be exactly one solution, and hence exactly one negative eigenvalue. Since $\phi$ was determined by $\lambda$ up to choosing $b_0$, the claim follows.
\end{proof}

Denote by $\phi_0$ the unique solution to \eqref{eq:radial eigenfunction higher eigenvalues} with $\mu=\mu_0$ and $\phi_0(0)=1$ and set $\lambda_0\doteqdot\sqrt{-\mu_0}$. Note that
\begin{equation}\label{eq:defining equation lambda_0}
\lambda_0\frac{\phi'(\lambda_0)}{\phi(\lambda_0)}=1
\end{equation}
and, since $\lambda_0$ is increasing in the dimension $n$, $\lambda_0\ge \lambda_0|_{n=1}>1$. 

\begin{lemma}\label{lem:negative eigenspace}
The null eigenspace of \eqref{eq:eigenfunction} consists of the linear functions.
\end{lemma}
\begin{proof}
When $\mu=0$, the equation \eqref{eq:radial eigenfunction higher eigenvalues} can be solved directly. Its admissible (i.e. square integrable) solutions are multiples of
\[
\phi_\ell(r)\doteqdot r^\ell\,,
\]
but the boundary condition rules out all but the linear ones, $\ell=1$. The only admissible spherical harmonics $\Phi$ are then those of degree $\ell=1$.
\end{proof}

\begin{corollary}\label{cor:Liouville}
Let 
$u$ be a non-negative ancient solution to the critical-Robin heat equation on the unit ball. If $u(\cdot,t)=\mathrm{e}^{o(-t)}$ as $t\to-\infty$, then
\[
u(x,t)=A\mathrm{e}^{\lambda_0^2t}\phi_0(\lambda_0\vert x\vert)\;\;\text{for some}\;\; A\in\R\,.
\]
\end{corollary}
\begin{proof}
We may represent $u$ as
\[
u(x,t)=\sum_{j=0}^\infty A_j\mathrm{e}^{-\mu_j t}\phi_j(x)\,,
\]
where $\{\phi_j\}_{j=0}^\infty$ is an orthonormal basis for $L^2(B^n)$ consisting of eigenfunctions $\phi_j$ of the critical-Robin Laplacian (with eigenvalue $\mu_j$). Since the null modes are linear and $\mu_j>0$ for each $j\ge n+1$, the coefficients of these states must be zero.
\end{proof}

\subsection{Further properties of the ground state}\label{sec:properties of phi}

Recall that the radial ground state $\phi_0$ is given by
\[
\phi_0(\rho)=1+\sum_{j=1}^\infty b_j\rho^{2j}\,,\;\;\text{where}\;\; b_j=\prod_{\ell=1}^j\frac{1}{2\ell(n+2(\ell-1))}\,.
\]
In particular, $\phi_0$ well-defined and analytic on the whole real line, and satisfies
\[
\phi_0''+\frac{n-1}{\rho}\phi_0'=\phi_0
\]
everywhere.\footnote{In the following, it is instructive to keep in mind that, when $n=1$, $\phi_0(\rho)=\cosh\rho$.} We shall need the following basic properties of $\phi_0$.

\begin{lemma}\label{lem:first phi lemma}
The (odd) function $\phi_0'/\phi_0$ is monotone increasing 
and converges to one as $\rho\to\infty$.
\end{lemma}
\begin{proof}
Setting $\Phi\doteqdot \phi_0'/\phi_0$, we find that
\[
\Phi(\rho)
=\frac{\sum_{j=0}^\infty 2(j+1)b_{j+1}\rho^{2j+1}}{\sum_{j=0}^\infty b_j\rho^{2j}}
\sim\frac{\rho}{n}\;\;\text{for}\;\;\rho\sim 0\,.
\]
Observe also that
\[
\rho\Phi'=\rho(1-\Phi^2)-(n-1)\Phi
\]
and
\[
\rho\Phi''=1-\Phi^2-n\Phi'-2\rho\Phi\Phi'\,.
\]
The first observation implies that $\Phi(\rho)>0=\Phi(0)$ for $\rho$ close to (but above) zero. The second implies that $\Phi$ can never reach one, and the third then implies that 
$\Phi'$ remains positive for all $\rho>0$. The second observation then implies that $\Phi(\rho)\to 1$ as $\rho\to\infty$.
\end{proof}

\begin{lemma}\label{lem:second phi lemma}
The (even) function $\phi'_0/\rho\phi_0$ is monotone decreasing for $\rho>0$ and bounded from above by $1/n$.
\end{lemma}
\begin{proof}
If we set $\Phi\doteqdot\phi'_0/\rho\phi_0$, then
\bann
\Phi(\rho)=\frac{\sum_{j=0}^\infty 2(j+1)b_{j+1}\rho^{2j}}{\sum_{j=0}^\infty b_j\rho^{2j}}
\eann
and
\bann
\rho\Phi''={}&-(n+1)\Phi'-2\rho\Phi^2-2\rho^2\Phi\Phi'\,.
\eann
The first observation implies that $\Phi(0)=2b_1=\frac{1}{n}$ and hence
\bann
\Phi(\rho)-\Phi(0)={}&\frac{\sum_{j=0}^\infty 2(j+1)b_{j+1}\rho^{2j}-2b_1\sum_{j=0}^\infty b_j\rho^{2j}}{\sum_{j=0}^\infty b_j\rho^{2j}}\\
={}&\frac{2\sum_{j=0}^\infty \big((j+1)b_{j+1}-b_1b_j\big)\rho^{2j}}{\sum_{j=0}^\infty b_j\rho^{2j}}\\
<{}&0\,.
\eann
The second observation then implies that $\Phi$ remains decreasing for all $\rho>0$.
\end{proof}

\section{Existence}\label{sec:existence}

We are now ready to construct a non-trivial ancient free boundary mean curvature flow in the $(n+1)$-ball. It will be clear from the construction that the solution is rotationally symmetric about the vertical axis, emerges at time negative infinity from the horizontal $n$-ball, and converges at time zero to the point $e_{n+1}$. 

We shall also prove an estimate for the height of the constructed solution (which will be needed to prove its uniqueness).


\subsection{Barriers}


Given $\theta\in (0,\frac{\pi}{2})$, denote by $\mathrm{S}_\theta$ the $n$-sphere centred on the $e_{n+1}$-axis which meets $\pd B^{n+1}$ orthogonally at $\{\cos\theta e+\sin\theta e_{n+1}:e\in S^{n-1}\times \{0\}\}$. That is,
\begin{equation}\label{eq:critical circle}
\mathrm{S}_\theta\doteqdot \left\{(x,y)\in \R^{n}\times\R:\vert x\vert^2+\left(\csc\theta-y\right)^2=\cot^2\theta\right\}\,.
\end{equation}
If we set
\[
\theta^-(t)\doteqdot \arcsin \mathrm{e}^{nt}\;\;\text{and}\;\; \theta^+(t)\doteqdot \arcsin \mathrm{e}^{2nt}\,,
\]
then we find that the inward normal speed of $\mathrm{S}_{\theta^-(t)}$ is no greater than its mean curvature $H^-$ and the inward normal speed of $\mathrm{S}_{\theta^+(t)}$ is no less than its mean curvature $H^+$. The maximum principle and the Hopf boundary point lemma then imply that 
\begin{proposition}\label{prop:circle barriers}
A solution to free boundary mean curvature flow in $B^{n+1}$ which lies below (resp. above) the sphere $\mathrm{S}_{\theta_0}$ at time $t_0$ lies below $\mathrm{S}_{\theta^+(t^+_0+t-t_0)}$ (resp. above $\mathrm{S}_{\theta^-(t^-_0+t-t_0)}$) for all $t>t_0$, where $2nt^+_0=\log\sin\theta_0$ (resp. $nt^-_0=\log\sin\theta_0$).
\end{proposition}

Next, given $\lambda>0$, consider the family $\{\Sigma^\lambda_t\}_{t\in(-\infty,0)}$ of hypersurfaces $\Sigma_t^\lambda$ defined by
\[
\Sigma^\lambda_t\doteqdot \left\{(x,y)\in\R^n\times(0,\tfrac{\pi}{2\lambda}):\sin(\lambda y)=\mathrm{e}^{\lambda^2t}\phi\big(\lambda\vert x\vert\big)\right\}\,,
\]
where $\phi=\phi_0$ is the unique solution to \eqref{eq:negative eigenspace change of variable} with $\mu=\mu_0$ and $\phi(0)=1$. 

Differentiating the defining equation with respect to an arclength parameter $s$ along the profile curve $\Lambda^\lambda_t\doteqdot \Sigma^\lambda_t\cap e_1\wedge e_{n+1}\cap\{x_1>0\}$ yields (with the abuse of notation $x=x_1$)
\begin{equation}\label{eq:tangent for phi curve}
\cos(\lambda y)y_s=\mathrm{e}^{\lambda^2t}\phi'(\lambda x)x_s\,.
\end{equation}
From this we find that the downwards pointing normal to the profile curve is given by
\[
\nu=\left.{\left(\frac{\phi'(\lambda x)}{\phi(\lambda x)},-\frac{\cos(\lambda y)}{\sin(\lambda y)}\right)}\right/{\sqrt{\left(\frac{\phi'(\lambda x)}{\phi(\lambda x)}\right)^2+\left(\frac{\cos(\lambda y)}{\sin(\lambda y)}\right)^2}}\,.
\]
Differentiating \eqref{eq:tangent for phi curve} yields
\bann
\cos(\lambda y)y_{ss}-\mathrm{e}^{\lambda^2t}\phi'{}&(\lambda x)x_{ss}=\lambda\left(\sin(\lambda y)y_s^2+\mathrm{e}^{\lambda^2t}\phi''(\lambda x)x_s^2\right)\\
={}&\lambda\left(\sin(\lambda y)y_s^2+\mathrm{e}^{\lambda^2t}\left[\phi(\lambda x)-\frac{n-1}{\lambda x}\phi'(\lambda x)\right]x_s^2\right).
\eann
From this we find that the curvature of the profile curve is given by
\bann
\kappa={}&-\gamma_{ss}\cdot\nu\\
={}&\lambda\tan(\lambda y)x_s\left(1-\frac{n-1}{\lambda x}\frac{\phi'(\lambda x)}{\phi(\lambda x)}x_s^2\right).
\eann
On the other hand, the rotational curvature $\hat\kappa$ is given by
\bann
\hat\kappa={}&\frac{\nu\cdot e_1}{x}=\tan(\lambda y)\frac{x_s}{x}\frac{\phi'(\lambda x)}{\phi(\lambda x)}\,.
\eann
If we choose the arclength parameter so that $s=0$ corresponds to $x=0$, then we may write $\tau\doteqdot \gamma_s=(\cos\theta,\sin\theta)$ with $\theta\in[0,\frac{\pi}{2})$ increasing in $s$. The mean curvature of $\Sigma^\lambda_t$ is then given, along the profile curve $\Lambda^\lambda_t$, by
\[
H=\lambda\tan(\lambda y)\cos\theta\left(1+(n-1)\sin^2\theta\frac{\phi'(\lambda x)}{\lambda x\phi(\lambda x)}\right).
\]
On the other hand,
\[
\cos(\lambda y)y_t=\mathrm{e}^{\lambda^2t}\phi'(\lambda x)x_t+\lambda\mathrm{e}^{\lambda^2t}\phi(\lambda x)
\]
so that
\[
-\gamma_t\cdot\nu=\lambda\tan(\lambda y)\cos\theta\le H\,.
\]
That is, $\{\Sigma^\lambda_t\}_{t\in(-\infty,0)}$ is a subsolution to mean curvature flow.

\begin{proposition}\label{prop:Angenent oval barriers}
Given $\theta\in (0,\frac{\pi}{2})$, there exists a unique solution $\lambda(\theta)$ to
\[
\frac{\phi'(\lambda \cos\theta)}{\phi(\lambda \cos\theta)}\cos\theta=\frac{\cos(\lambda \sin\theta)}{\sin(\lambda \sin\theta)}\sin\theta\,.
\]
Let $\{M_t\}_{t\in [\alpha,\omega)}$ be a rotationally symmetric solution to free boundary mean curvature flow in $B^{n+1}$. Define $\theta_\alpha\in(0,\frac{\pi}{2})$ by
\[
\sin\theta_\alpha=\{X\cdot e_{n+1}:X\in \pd M_\alpha\}\,.
\]
If $\lambda\le\lambda(\theta_\alpha)$ and $M_\alpha$ lies above $\Sigma^{\lambda}_{s}$, then $M_t$ lies above $\Sigma^{\lambda}_{s+t-\alpha}$ for all $t\in(\alpha,\omega)\cap(-\infty,\alpha-s)$.
\end{proposition}
\begin{proof}
Due to Lemma \ref{lem:first phi lemma}, we may proceed as in \cite[Lemma 2.2 and Proposition 2.3]{FBCSF}.
\end{proof}
Note also that, by the defining property \eqref{eq:defining equation lambda_0} of $\lambda_0$, we have the inequality $\lambda(\theta)<\lambda_0$.

We will also require the following technical properties of $\Sigma^\lambda$.
\begin{lemma}\label{lem:gradient conds on initial surface}
The mean curvature $H$ of the hypersurface $\Sigma^\lambda$ satisfies
\[
\cd H\cdot X\ge 0\;\;\text{and}\;\; \limsup_{\lambda\to\lambda_0}\max_{\overline \Sigma{}^\lambda}\vert\cd \log H\vert\le 1\,.
\]
\end{lemma}
\begin{proof}
If we set $\Phi(\rho)\doteqdot \phi'(\rho)/\rho\phi(\rho)$, then
\bann
\hat\kappa_s={}&\frac{\cos\theta}{x}(\kappa-\hat\kappa)\\
={}&\frac{\cos\theta}{x}\left(\lambda\tan(\lambda y)\cos\theta\left[1-n\Phi(\lambda x)+(n-1)\sin^2\theta\Phi(\lambda x)\right]\right)
\eann
and 
\bann
\kappa_s={}&\big(\lambda^2\sec^2(\lambda y)\cos\theta\sin\theta-\lambda\tan(\lambda y)\sin\theta \kappa\big)(1-(n-1)\cos^2\theta\Phi(\lambda x))\\
{}&+(n-1)\lambda\tan(\lambda y)\cos\theta\big(2\cos\theta\sin\theta\Phi(\lambda x)\kappa-\lambda\cos^3\theta\Phi'(\lambda x)\big)\\
={}&\big(\lambda^2\cos\theta\sin\theta+(n-1)\lambda^2\tan^2(\lambda y)\sin\theta\cos^3\theta\Phi(\lambda x)\big)\\
{}&\hspace{200pt}\cdot(1-(n-1)\cos^2\theta\Phi(\lambda x))\\
{}&+(n-1)\lambda\tan(\lambda y)\cos\theta\big(2\cos\theta\sin\theta\Phi(\lambda x)\kappa-\lambda\cos^3\theta\Phi'(\lambda x)\big)\,.
\eann
These are both positive for $\theta>0$ by Lemma \ref{lem:second phi lemma}, which implies the first claim.

Next, estimating (in the \emph{second} term below) 
\[
\kappa\le H=\lambda\tan(\lambda y)\cos\theta(1+(n-1)\sin^2\theta\Phi(\lambda x))
\]
and $\Phi'(\rho)\le 0$ for $\rho\ge 0$, we observe that 
\bann
(\log H)_s={}&\log\big(\lambda\tan(\lambda y)\cos\theta\big)_s+\log\big(1+(n-1)\sin^2\theta\Phi(\lambda x)\big)_s\\
={}&\frac{\lambda^2(1+\tan^2(\lambda y))\cos\theta\sin\theta-\lambda\tan(\lambda y)\sin\theta\kappa}{\lambda\tan(\lambda y)\cos\theta}\\
{}&+(n-1)\frac{2\sin\theta\cos\theta\kappa\Phi(\lambda x)+\lambda\sin^2\theta\cos\th\Phi'(\lambda x)}{1+(n-1)\sin^2\theta\Phi(\lambda x)}\\
\le{}&
\frac{\lambda\sin\theta}{\tan(\lambda y)}+(n-1)\lambda\tan(\lambda y)\cos^2\theta\sin\theta\Phi(\lambda x)\\
{}&+2(n-1)\lambda\tan(\lambda y)\sin\theta\cos^2\theta\Phi(\lambda x)\,.
\eann
The second and third terms approach zero uniformly as $\lambda\to\lambda_0$ since $\theta$ does. The first may be rewritten as
\[
\frac{\lambda\sin\theta}{\tan(\lambda y)}=\frac{\phi'(\lambda x)}{\phi(\lambda x)}\lambda\cos\theta \,.
\]
The second claim now follows from Lemma \ref{lem:first phi lemma} and the identity \eqref{eq:defining equation lambda_0}.
\end{proof}

\subsection{Old-but-not-ancient solutions}
Given $\rho>0$, choose a hypersurface $M^\rho$ in $B^{n+1}$ which satisfies the following properties:
\begin{itemize}
\item $M^\rho$ is rotationally symmetric about the $e_{n+1}$-axis,
\item $M^\rho$ meets $\pd B^{n+1}$ orthogonally at $\{\cos\rho e+\sin\rho e_{n+1}:e\in S^{n-1}\times\{0\}\}$,
\item $M^\rho\cap B^{n+1}$ is the relative boundary of a convex region $\Omega^\rho\subset B^{n+1}$, and
\item $\cd H^\rho\cdot X\ge 0$, where $X$ denotes the position vector.
\end{itemize}
For example, we could take $M^\rho\doteqdot \Sigma^{\lambda_\rho}_{t_\rho}$, where $\lambda_\rho\doteqdot \lambda(\rho)$ and
\[
-t_\rho=\lambda_\rho^{-2}\log\left(\frac{\phi(\lambda_\rho\cos\rho)}{\sin(\lambda_\rho\sin\rho)}\right)\,.
\]

Recall ing the notation from \eqref{eq:critical circle}, observe that the sphere $\mathrm{S}_{\theta_\rho}$ defined by
\[
\sin\theta_\rho=\frac{2\sin\rho}{1+\sin^2\rho}
\]
is tangent to the plane $\{X\in \R^{n+1}:X\cdot e_{n+1}=\sin\rho\}$, and hence lies above $M^\rho$.

Work of Stahl \cite{Stahl2,Stahl1} now yields the following \emph{old-but-not-ancient solutions}.

\begin{lemma}\label{lem:old-but-not-ancient}
For each $\rho\in(0,\frac{\pi}{2})$, there exists a smooth solution\footnote{Given by a one parameter family of immersions $X:\overline M\times [\alpha_\rho,0)\to \overline B{}^{n+1}$ satisfying $X\in C^\infty(\overline M\times (\alpha_\rho,0))\cap C^{2+\beta,1+\frac{\beta}{2}}(\overline M\times [\alpha_\rho,0))$ for some $\beta\in(0,1)$.} $\{M^\rho_t\}_{t\in[\alpha_\rho,0)}$ to free boundary mean curvature flow in $B^{n+1}$ which satisfies the following properties:
\begin{itemize}
\item $M^\rho_{\alpha_\rho}=M^\rho$,
\item $M^\rho_t$ is convex and locally uniformly convex for each $t\in(\alpha_\rho,0)$,
\item $M^\rho_t$ is rotationally symmetric about the $e_{n+1}$-axis for each $t\in(\alpha_\rho,0)$,
\item $M_t^\rho\to e_{n+1}$ uniformly as $t\to 0$,
\item $\cd H^\rho\cdot X\ge 0$, and
\item $\alpha_\rho<\frac{1}{2n}\log\left(\frac{2\sin\rho}{1+\sin^2\rho}\right)\to-\infty$ as $\rho\to 0$.
\end{itemize}
\end{lemma}
\begin{proof}
Existence of a maximal solution to mean curvature flow out of $M^\rho$ which meets $\pd B^{n+1}$ orthogonally was proved by Stahl \cite[Theorem 2.1]{Stahl1}. Stahl also proved that this solution remains convex and locally uniformly convex and shrinks to a point on the boundary of $B^{n+1}$ at the final time (which is finite) \cite[Proposition 1.4]{Stahl2}. We obtain $\{M_t^\rho\}_{t\in[\alpha_\rho,0)}$ by time-translating Stahl's solution. 

By uniqueness of solutions (or the Alexandrov reflection principle) $M^\rho_t$ remains rotationally symmetric about the $e_{n+1}$-axis for $t\in(\alpha_\rho,0)$, so the final point is $e_{n+1}$.

The rotational symmetry also implies that $\cd H=0$ at the point $p_t\doteqdot M^\rho_t\cap\R e_{n+1}$ for all $t\in[\alpha_\rho,0)$. By \cite[Proposition 2.1]{Stahl2}, $\cd H\cdot X=H>0$ at the boundary for all $t\in(\alpha_\rho,0)$. The maximum principle now implies that $\cd H\cdot X\ge 0$. To see this, recall that
\[
(\pd_t-\Delta)\vert \cd H\vert\le c\vert A\vert^2\vert \cd H\vert
\]
where $c$ is a constant that depends only on $n$, and, given any $\sigma\in (\alpha_\rho,0)$ and any $\varepsilon>0$, consider the function
\[
v_{\sigma,\varepsilon}\doteqdot\left\{\begin{aligned}\inner{\cd H}{\vec r}+\varepsilon\mathrm{e}^{(C_\sigma c+1)(t-\alpha_\rho)}\;\;\text{on}\;\; B^{n+1}\setminus \R e_{n+1}\\
\varepsilon\mathrm{e}^{(C_\sigma c+1)(t-\alpha_\rho)}\;\;\text{on}\;\; \R e_{n+1}\,,\end{aligned}\right.
\]
where $C_\sigma\doteqdot\max_{t\in [\alpha_\rho,\sigma]}\max_{\overline M{}_t^\rho}\vert A\vert^2$ and $\vec r\doteqdot X^\top/\vert X^\top\vert$. Note that $v_{\sigma,\varepsilon}$ is continuous. Observe that $v_{\sigma,\varepsilon}$ is no less than $\varepsilon$ on $\pd M_t^\rho$ for all $t$, on the axis off rotation for all $t$, and everywhere at the initial time. Thus, if $v_{\sigma,\varepsilon}$ is not positive everywhere in $\overline M^\rho\times[\alpha_\rho,\sigma]$, then there must be a first time $t_0\in(0,\sigma]$ and an off-axis interior point $x_0$ such that $v_{\sigma,\varepsilon}(x_0,t_0)=0$ and $v_{\sigma,\varepsilon}\ge 0$ in a small parabolic neighbourhood of $(x_0,t_0)$. Since $\inner{\cd H}{\vec r}|_{(x_0,t_0)}<0$, we find that $\inner{\cd H}{\vec r}=-\vert \cd H\vert$ in a small spacetime neighborhood of $(x_0,t_0)$. But then, at $(x_0,t_0)$,
\bann
0\ge{}& (\pd_t-\Delta)v_{\sigma,\varepsilon}\\
\ge{}&-C_\sigma c\vert \cd H\vert+\varepsilon(C_\sigma c+1)\mathrm{e}^{(C_\sigma c+1)(t-\alpha_\rho)}\\
={}&\varepsilon\mathrm{e}^{(C_\sigma c+1)(t-\alpha_\rho)}\\
>{}&0\,,
\eann
which is absurd. So $v_{\varepsilon,\sigma}$ is indeed positive in $\overline M{}_t^\rho$ for all $t\in [\alpha_\rho,\sigma]$. Taking $\varepsilon\to 0$ and $\sigma\to 0$ then implies that $\cd H\cdot X\ge 0$.

Since $\mathrm{S}_{\theta_\rho}\subset\Omega^\rho$, the final property follows from Proposition \ref{prop:circle barriers}.
\end{proof}

We now fix $\rho>0$ and drop the super/subscript $\rho$. Denote by $\Gamma_t=M_t\cap e_1\wedge e_{n+1}\cap\{x_1>0\}$ the profile curve of $M_t$ and 
%
%
set
\[
p_t\doteqdot M_t\cap \R e_{n+1}\,,\;\; q_t\doteqdot \pd B^{n+1}\cap \Gamma_t\,,
\]
\[
\underline H(t)\doteqdot \min_{M_t} H= H(p_t)\;\;\text{and}\;\;\overline H(t)\doteqdot \max_{M_t} H= H(q_t)\,,
\]
and define $\underline y(t)$, $\overline y(t)$ and $\overline\theta(t)$ by
\[
p_t=\underline y(t)e_{n+1}\,,\;\;q_t=\cos\overline\theta(t)e_1+\sin\overline\theta(t)e_{n+1}\,,\;\;\text{and}\;\;\overline y(t)=\sin\overline\theta(t)\,.
\]

\begin{lemma}\label{lem:crude estimates}
Each old-but-not-ancient solution satisfies
\begin{equation}\label{eq:kappa bounds}
\underline H\le n\tan\overline\theta\le \overline H\,,
\end{equation}
\begin{equation}\label{sin theta upper bound}
\sin\overline\theta\le e^{nt}\,,
\end{equation}
and
\begin{equation}\label{eq:min y lower bound}
\frac{\sin\overline\theta}{1+\cos\overline\theta}\le \underline y\le \sin\overline\theta\,.
\end{equation}
\end{lemma}
\begin{proof}
To prove the lower bound for $\overline H$, it suffices to show that the sphere $\mathrm{S}_{\overline \theta(t)}$ (see \eqref{eq:critical circle}) lies locally below $M_t$ near $q_t$. If this is not the case, then, locally around $q_t$, $M_t$ lies below $\mathrm{S}_{\overline \theta(t)}$ and hence $H(q_t)\le n\tan\overline\theta(t)$. But then we can translate $\mathrm{S}_{\overline \theta(t)}$ downwards until it touches $M_t$ from below in an interior point at which the curvature must satisfy $H\ge n\tan\overline\theta(t)$. This contradicts the maximization of the mean curvature at $q_t$ (unless $M_t$ coincides with $\mathrm{S}_{\overline \theta(t)}$ in a neighbourhood of $q_t$, which by itself implies the claim).

The estimate \eqref{sin theta upper bound} now follows by integrating the inequality
\[
\frac{d}{dt}\sin\overline\theta=\cos\overline\theta\,\overline H\ge n\sin\overline\theta
\]
between any initial time $t$ and the final time $0$ (at which $\overline\theta=\frac{\pi}{2}$ since the solution contracts to the point $e_{n+1}$).

The upper bound for $\underline y$ follows from convexity and the boundary condition $\overline y=\sin\overline\theta$. To prove the lower bound, we will show that the sphere $\mathrm{S}_{\overline \theta(t)}$ lies nowhere above $M_t$. Suppose that this is not the case. Then, since $\mathrm{S}_{\overline \theta(t)}$ lies locally below $M_t$ near $q_t$, we can move $\mathrm{S}_{\overline\theta(t)}$ downwards until it is tangent from below to a point $p'_t$ on $M_t$, at which we must have $H\ge n\tan\overline\theta(t)$. But then, since $\cd H\cdot X\ge 0$, we find that $H\ge n\tan \overline\theta(t)$ for all points on the radial curve between $p'_t$ and the nearest boundary point. But this implies that this whole arc (including $p'_t$) lies above $\mathrm{S}_{\overline\theta(t)}$, a contradiction.

To prove the upper bound for $\underline H$, fix $t$ and consider the sphere $S$ centred on the $e_{n+1}$-axis through the points $p_t$ and $q_t$. Its radius is $r(t)$, where
\[
r\doteqdot \frac{\cos^2\overline\theta+(\sin\overline\theta-\underline y)^2}{2(\sin\overline\theta-\underline y)}\,.
\]
We claim that $M_t$ lies locally below $S$ near $p_t$. Suppose that this is not the case. Then, by the rotational symmetry of $M_t$ and $S$ about the $e_{n+1}$-axis, $M_t$ lies locally above $S$ near $p_t$. This implies two things: first, that
\[
 H(p_t)\ge nr^{-1},
\]
and second, that, by moving $S$ vertically upwards, we can find a point $p'_t$ (the final point of contact) which satisfies
\[
 H(p'_t)\le nr^{-1}\,.
\] 
These two inequalities contradict the (unique) minimization of $H$ at $p_t$. We conclude that
\[
\underline H\le\frac{2n(\sin\overline\theta-\underline y)}{\cos^2\overline\theta+(\sin\overline\theta-\underline y)^2}\le n\tan\overline\theta
\]
due to the lower bound for $\underline y$.
\end{proof}

\subsection{Taking the limit}

\begin{proposition}\label{prop:existence}
There exists a convex, locally uniformly convex ancient mean curvature flow in $B^{n+1}$ with free boundary on $\pd B^{n+1}$.
\end{proposition}
\begin{proof}
For each $\rho>0$, consider the old-but-not-ancient solution $\{M^\rho_t\}_{t\in[\alpha_\rho,0)}$, $M^\rho_t=\pd \Omega^\rho_t$, constructed in Lemma \ref{lem:old-but-not-ancient}. By \eqref{sin theta upper bound}, $\Omega^\rho_t$ contains $\mathrm{S}_{\omega(t)}\cap B^{n+1}$, where $\omega(t)\in(0,\frac{\pi}{2})$ is uniquely defined by
\[
\frac{1-\cos\omega(t)}{\sin\omega(t)}=\mathrm{e}^{nt}\,.
\]
If we represent $M^\rho_t$ as a graph $x\mapsto y^\rho(x,t)$ over the horizontal $n$-ball, then convexity and the boundary condition imply that $\vert Dy^\rho\vert\le \tan\omega$. Since $\omega(t)$ is independent of $\rho$, Stahl's (global in space, interior in time) Ecker--Huisken type estimates \cite{Stahl1} imply uniform-in-$\rho$ bounds for the curvature and its derivatives. So the limit
\[
\{M^\rho_t\}_{t\in[\alpha_\rho,0)}\to \{M_t\}_{t\in(-\infty,0)}
\]
exists in $C^\infty$ (globally in space on compact subsets of time) and the limit $\{M_t\}_{t\in(-\infty,0)}$ satisfies mean curvature flow with free boundary in $B^{n+1}$. On the other hand, since  $\{M^\rho_t\}_{t\in(\alpha_\rho,0)}$ contracts to $e_{n+1}$ as $t\to 0$, (the contrapositive of) Proposition \ref{prop:circle barriers} implies that $M^\rho_t$ must intersect the closed region enclosed by $\mathrm{S}_{\theta^+(t)}$ for all $t<0$. It follows that $M_t$ converges to $e_{n+1}$ as $t\to 0$. By \cite[Proposition 4.5]{Stahl2}, $M_t$ is locally uniformly convex for each $t$. Since each $M_t$ is the limit of convex hypersurfaces, each is convex.
\end{proof}

\subsection{Asymptotics for the height}

For the purposes of this section, we fix an ancient solution $\{M_t\}_{(-\infty,0)}$ obtained as in Proposition \ref{prop:existence} by taking a sublimit as $\lambda\searrow \lambda_0$ of the specific old-but-not ancient solutions $\{M^\lambda_t\}_{t\in[\alpha_\lambda,0)}$ corresponding to $M^\lambda_{\alpha_\lambda}=\Sigma^{\lambda}_{t_\lambda}\cap B^{n+1}$, $t_\lambda$ being the time at which $\{\Sigma^\lambda_t\}_{t\in(-\infty,0)}$ meets $\pd B^{n+1}$ orthogonally. The asymptotics we obtain for this solution will be used to prove its uniqueness. 

We will need to prove that the limit $\lim_{t\to-\infty}\mathrm{e}^{-\lambda_0^2t}\underline y(t)$ exists in $(0,\infty)$. The following speed bound will imply that it exists in $[0,\infty)$. 

\begin{lemma} \label{lem:lower curvature estimate}
The ancient solution $\{M_t\}_{(-\infty,0)}$ satisfies
\begin{equation}\label{eq:curvature lower}
\underline H\ge \lambda_0^2\underline y\,.
\end{equation}
\end{lemma}
\begin{proof}
It suffices to prove that
\begin{equation}\label{eq:curvature lower old but not ancient}
\min_{M^\lambda_t}\frac{H}{y}\ge\min_{t=\alpha_\lambda}\frac{H}{y}
\end{equation}
on each of the old-but-not-ancient solutions $\{M^\lambda_t\}_{t\in[\alpha_\lambda,0)}$ since
\ba
\min_{t=\alpha_\lambda}\frac{H}{y}={}&\min_{\Sigma^{\lambda}_{t_\lambda}}\frac{H}{y}\nonumber\\
={}&\lambda^2\min_{\Sigma^{\lambda}_{t_\lambda}}\frac{\tan(\lambda y)}{\lambda y}\cos\theta\left(1+(n-1)\sin^2\theta\frac{\phi'(\lambda x)}{\lambda x\phi(\lambda x)}\right)\nonumber\\
\ge{}&\lambda^2\cos\theta_\lambda\nonumber\\
\to{}&\lambda_0^2\;\;\text{as}\;\;\lambda\to\lambda_0\,,\label{eq:curvature bound initial data}
\ea
where $\theta_\lambda$ is the angle that $\Lambda^\lambda_t$ meets the boundary. But this is an easy consequence of the maximum principle, since
\bann
(\pd_t-\Delta)\frac{H}{y}={}&\vert A\vert^2\frac{H}{y}+2\inner{\cd\frac{H}{y}}{\frac{\cd y}{y}}\\
\ge{}& 2\inner{\cd\frac{H}{y}}{\frac{\cd y}{y}}
\eann
in the interior and
\[
\cd\frac{H}{y}=0
\]
at the boundary.
\end{proof}

It follows that
\begin{equation}\label{eq:rescaled height monotonicity}
\big(\mathrm{e}^{-\lambda_0^2t}\underline y\big)_t\ge 0\,.
\end{equation}
In particular, the limit
\[
A\doteqdot\lim_{t\to-\infty}\mathrm{e}^{-\lambda_0^2t}\underline y(t)
\]
exists in $[0,\infty)$ as claimed.

We want next to prove that the above limit is positive. We will achieve this through a suitable upper bound for the speed.

Recall that
\begin{equation}\label{eq:evolution equations}
(\pd_t-\Delta)\vert\cd\!H\vert\le c\vert A\vert^2\vert\cd\!H\vert\;\;\text{and}\;\;(\partial_t-\Delta)\langle X,\nu\rangle=\vert A\vert^2\!\langle X, \nu\rangle-2H,
\end{equation}
where $c$ depends only on $n$ and $X$ denotes the position. The good $-2H$ term in the second equation may be exploited to obtain the following crude speed bound.

\begin{lemma}\label{lem:gradient estimate}
There exist $T>-\infty$ and $C<\infty$ such that
\begin{equation}\label{eq:curvature estimate}
\overline H\le Ce^{nt}\;\;\text{for all}\;\; t<T\,.
\end{equation}
\end{lemma}
\begin{proof}
We will prove the estimate on the (sufficiently) old-but-not-ancient solutions $\{M^\lambda_t\}_{t\in(\alpha_\lambda,0)}$. We first prove a crude gradient estimate of the form
\begin{equation}\label{eq:crude kappa gradient estimate}
\vert\cd H\vert\le 4H
\end{equation}
for $t$ sufficiently negative. It will suffice to prove that
\begin{equation}\label{eq:1k}
\vert\cd H\vert-2H+2\inner{X}{\nu}\le 0\,,
\end{equation}
where $X$ denotes the position. Indeed, since
\[
\inner{\cd\inner{X}{\nu}}{X}=A(X^\top,X^\top)>0\,,
\]
we may estimate, by \eqref{eq:curvature lower old but not ancient} and \eqref{eq:curvature bound initial data},
\begin{equation}\label{eq:gamma dot nu le kappa}
\vert\!\inner{X}{\nu}\!\vert\le\vert\!\inner{X}{\nu}\!\vert_{x=0}=y|_{x=0}\le H|_{x=0}=\min_{M^\lambda_t}H\le H
\end{equation}
for $\lambda$ sufficiently close to $\lambda_0$.

For $\lambda$ sufficiently close to $\lambda_0$, we have $H|_{t=\alpha_\lambda}<1/\sqrt c$, where $c$ is the constant in \eqref{eq:evolution equations}. Denote by $T^\lambda$ the first time at which $H$ reaches $1/\sqrt c$. Since $H$ is continuous up to the initial time $\alpha_\lambda$, we have $T^\lambda>\alpha_\lambda$. We claim that \eqref{eq:1k} holds for $t<T^\lambda$ so long as $\lambda$ is sufficiently large. It is satisfied on the initial timeslice $M^\lambda_{\alpha_\lambda}=\Sigma^\lambda_{t_\lambda}$ by Lemma \ref{lem:gradient conds on initial surface}. We will show that the function
\[
f_\varepsilon\doteqdot\vert\cd H\vert-2H+2\inner{X}{\nu}-\varepsilon\mathrm{e}^{t-\alpha_\lambda}
\]
remains negative up to time $T^\lambda$. Suppose, to the contrary, that $f_\varepsilon$ reaches zero at some time $t<T^\lambda$ at some point $p\in\overline M{}^\lambda_t$. Since $\vert\cd H\vert-2H+2\inner{X}{\nu}$ is negative at the boundary, $p$ must be an interior point. 
At such a point, using the evolution equations \eqref{eq:evolution equations}, we have
\bann
0\le (\pd_t-\Delta)f_\varepsilon\le{}&\vert A\vert^2(c\vert\cd H\vert-2H+2\inner{X}{\nu})-4H-\varepsilon\mathrm{e}^{t-\alpha_\lambda}\\
={}&\vert A\vert^2\big(2(c-1)[H-\inner{X}{\nu}]+c\varepsilon\mathrm{e}^{t-\alpha_\lambda}\big)-4H-\varepsilon\mathrm{e}^{t-\alpha_\lambda}\,.
\eann
Recalling \eqref{eq:gamma dot nu le kappa} and estimating $\vert A\vert\le H$ and  $H\le \frac{1}{\sqrt{c}}$ yields
\[
0\le 4(c-1)H^3-4H+(cH^2-1)\varepsilon\mathrm{e}^{t-\alpha_\lambda}<0\,,
\]
which is absurd. So $f_\varepsilon$ does indeed remain negative, and taking $\varepsilon\to 0$ yields \eqref{eq:crude kappa gradient estimate} for $t<T^\lambda$.

Since the length of the profile curve $\Gamma^\lambda_t$ is bounded by $1$, integrating \eqref{eq:crude kappa gradient estimate} from the axis to the boundary yields
\[
\overline H\le \mathrm{e}^4\underline H\;\;\text{for}\;\; t<T^\lambda\,.
\]
Recalling \eqref{eq:kappa bounds} and \eqref{sin theta upper bound}, this implies that
\[
\overline H\le \mathrm{e}^4\frac{\mathrm{e}^{nt}}{\sqrt{1-\mathrm{e}^{2nt}}}\;\;\text{for}\;\; t<T^\lambda\,.
\]
Taking $t=T^\lambda$ we find that $T^\lambda\ge T$, where $T$ is independent of $\lambda$, so we conclude that
\[
\overline H\le C\mathrm{e}^{nt}\;\;\text{for}\;\; t<T\,,
\]
where $C$ and $T$ do not depend on $\lambda$.
\end{proof}

We now exploit \eqref{eq:curvature estimate} to obtain the desired speed bound.
\begin{lemma}
There exist $C<\infty$ and $T>-\infty$ such that
\[
\frac{H}{y}\le \lambda_0^2+C\mathrm{e}^{2nt}\;\;\text{for}\;\; t<T\,.
\]
\end{lemma}
\begin{proof}
Consider the old-but-not-ancient solution $\{M^\lambda_t\}_{t\in(-\infty,0)}$. By \eqref{eq:curvature estimate}, we can find $C<\infty$ and $T>-\infty$ such that
\bann
(\pd_t-\Delta)\frac{H}{y}={}&\vert A\vert^2\frac{H}{y}+2\inner{\cd\frac{H}{y}}{\frac{\cd y}{y}}\\
\le{}&C\mathrm{e}^{2nt}\frac{H}{y}+2\inner{\cd\frac{H}{y}}{\frac{\cd y}{y}}\;\;\text{for}\;\; t<T\,.
\eann
Since, at a boundary point,
\[
\cd\frac{H}{y}=\frac{\cd H}{y}-\frac{H}{y}\frac{\cd y}{y}=0\,,
\]
the Hopf boundary point lemma 
and the \textsc{ode} comparison principle yield
\[
\max_{M^\lambda_t}\frac{H}{y}\le C\max_{M^\lambda_{\alpha_\lambda}}\frac{H}{y}\;\;\text{for}\;\; t\in(\alpha_\lambda,T)\,.
\]
But now
\bann
(\pd_t-\Delta)\frac{H}{y}\le{}&C\mathrm{e}^{2nt}\max_{M^\lambda_{\alpha_\lambda}}\frac{H}{y}+2\inner{\cd\frac{H}{y}}{\frac{\cd y}{y}}\;\;\text{for}\;\; t<T\,,
\eann
and hence, by \textsc{ode} comparison,
\bann
\max_{M^\lambda_t}\frac{H}{y}\le \max_{M^\lambda_{\alpha_\lambda}}\frac{H}{y}\left(1+C\mathrm{e}^{2nt}\right)\;\;\text{for}\;\; t\in(\alpha_\lambda,T)\,.
\eann
Since, on the initial timeslice $M^\lambda_{\alpha_\lambda}=\Sigma_{t_\lambda}^\lambda$,
\[
\frac{H}{y}=\frac{\lambda\tan(\lambda y)}{y}\cos\theta\left(1+(n-1)\sin^2\theta\frac{\phi'(\lambda x)}{\lambda\phi(\lambda x)}\right)\,,
\]
the claim follows upon taking $\lambda\to\lambda_0$.
\end{proof}

It follows that
\[
\left(\log\underline y(t)-\lambda_0^2t\right)_t\le C\mathrm{e}^{2nt}\;\;\text{for}\;\; t<T
\]
and hence, integrating from time $t$ up to time $T$,
\[
\log\underline y(t)-\lambda_0^2t\ge\log\underline y(T)-\lambda_0^2T-C\;\;\text{for}\;\; t<T\,.
\]
So we indeed find that
\begin{lemma} \label{lem:backwards asymptotics}
the limit
\begin{equation}\label{eq:backwards asymptotics}
A\doteqdot\lim_{t\to-\infty}\mathrm{e}^{-\lambda_0^2t}\underline y(t)
\end{equation}
exists in $(0,\infty)$ on the particular ancient solution $\{M_t\}_{(-\infty,0)}$.
\end{lemma}

\section{Uniqueness}

Now let $\{M_t\}_{t\in(-\infty,0)}$, $M_t=\pd_{\mathrm{rel}}\Omega_t$, be \emph{any} convex, locally uniformly convex ancient free boundary mean curvature flow in the ball. By Stahl's theorem \cite{Stahl2}, we may assume that $M_t$ contracts to a point on the boundary as $t\to 0$.

\subsection{Backwards convergence}

We first show that $\overline M_t$ converges to a bisector as $t\to-\infty$.
\begin{lemma}\label{lem:backwards convergence}
Up to an ambient rotation,
\[
\overline M_t\underset{C^\infty}{\longrightarrow} \overline B{}^n\times\{0\}\,\,\text{as}\;\;t\to-\infty\,.
\]
\end{lemma}
\begin{proof}
Define $\Omega\doteqdot\cup_{t\in(-\infty,0)}\Omega_t$, where $\Omega_t\subset B^{n+1}$ is the convex region relatively bounded by $M_t$. Given $s\in\R$, define the free boundary mean curvature flow $\{M^s_t\}_{t\in(-\infty,-s)}$ by $M^s_t\doteqdot M_{t+s}=\pd\Omega^s_t$, where $\Omega^s_t\doteqdot \Omega_{t+s}$. Since the flow is monotone, the flows $\{M^s_{t}\}_{t\in(-\infty,-s)}$ converge to the stationary limit $\{\pd\Omega\}_{t\in(-\infty,\infty)}$ as $s\to-\infty$ uniformly in the Hausdorff topology on compact subsets of time. In fact, the convergence is smooth 
due to the Ecker--Huisken type estimates of Stahl \cite{Stahl2}. Now, since $\Omega$ is convex and its boundary intersects $\partial B^n$ orthogonally, it lies in some half-ball. But it cannot lie strictly in this half-ball due to Proposition \ref{prop:circle barriers}. The strong maximum principle and Hopf boundary point lemma then imply that $\Omega$ is a half-ball.
\end{proof}

We henceforth assume, without loss of generality, that the backwards limit is the horizontal $n$-ball.

\subsection{Reflection symmetry}

We can now prove that the solution is rotationally symmetric using Alexandrov reflection across planes through the origin (see Chow and Gulliver \cite{ChGu01}).

\begin{lemma}\label{lem:reflection}
$M_t$\! is rotationally symmetric about the $e_{n+1}\!$-axis for all $t$\!.
\end{lemma}
\begin{proof}
Given any $z\in S^n_+$ in the upper half sphere $S^n_+\doteqdot\{z\in S^n:z\cdot e_{n+1}>0\}$, we define the open halfspace
\[
H_z\doteqdot \{X\in \R^{n+1}:X\cdot z>0\}
\]
and denote by $R_z$ the reflection about $\partial H_z$. We first claim that, for every $z\in S^n_+$, there exists $t= t_z$ such that
\begin{equation}\label{eq:reflection}
(R_z\cdot M_t)\cap (M_t\cap H_z)=\emptyset
\end{equation}
for all $t<t_z$. Assume, to the contrary, that there exists $z\in S^n_+$, a sequence of times $t_i\to -\infty$, and a sequence of pairs of points $p_i, q_i\in M_{t_i}$ such that $R_z(p_i)= q_i$. This implies that the line passing through $p_i$ and $q_i$ is parallel to the vector $z$, so the mean value theorem yields for each $i$ a point $r_i$ on $M_{t_i}$ where the normal is orthogonal to $z$. But this contradicts Lemma \ref{lem:backwards convergence}.

The Alexandrov reflection principle \cite{ChGu01} now implies that \eqref{eq:reflection} holds for all $t<0$ (note that $R_z\cdot M_t$ also intersects $\partial B^{n+1}$ orthogonally). In fact, it is clear that $(R_z\cdot M_t)\cap H_z$ lies \emph{above} $M_t\cap H_z$ for all $t<0$ and all $z\in S^n_+$. The claim follows.
\end{proof}

\subsection{Asymptotics for the height}

We begin with a lemma.

\begin{lemma}
For every $t<0$,
\[
\cd H\cdot X\ge 0\;\;\text{in}\;\; M_t
\]
and hence
\begin{equation}\label{eq:min y lower bound general}
\frac{\sin\overline\theta}{1+\cos\overline\theta}\le \underline y\,.
\end{equation}
\end{lemma}
\begin{proof}
Choose $T>-\infty$ so that $H<\frac{2}{2c+1}$ for $t<T$, where $c=c(n)\ge 2$ is the constant in the evolution inequality for $\vert\cd H\vert$, and, given $\varepsilon>0$, define
\[
v_\varepsilon\doteqdot \inner{\cd H}{\vec r\,}+\e(1-\langle X, \nu\rangle)
\]
on $M_t\setminus \R e_{n+1}$, where $\vec r\doteqdot X^\top/\vert X^\top\vert$. Note that $v_\varepsilon\to \e(1-\langle X, \nu\rangle)$ as $X\to \R e_{n+1}$.

We claim that $v_\varepsilon\ge 0$ for $t\in(-\infty,T)$. Suppose that this is not the case. Since $v_\varepsilon(\cdot,t)>\varepsilon$ at $\pd M_t$, $v_\varepsilon(x,t)>\varepsilon$ as $x\to M_t\cap \R e_{n+1}$, and $v_\varepsilon\to \varepsilon$ as $t\to-\infty$, there must exist a first time $t\in (-\infty,T)$ and an interior point $x\in M_t\setminus \R e_{n+1}$ at which $v_\varepsilon=0$. But at such a point $\inner{\cd H}{\vec r\,}=-\e(1-\langle X, \nu\rangle)<0$, which means that $\inner{\cd H}{\vec r\,}=-\vert\cd H\vert$ in a small spacetime neighbourhood of $(p,t)$, and hence, at $(p,t)$,
\[
\begin{split}
0\ge{}&\left(\partial_t-\Delta\right)v_\varepsilon\\
\ge{}&\vert A\vert^2\big(\!\inner{\cd H}{\vec r\,}-\varepsilon\langle X, \nu\rangle\big)+(c-1)\vert A\vert^2\inner{\cd H}{\vec r\,}+2\varepsilon H\\
={}&-\varepsilon\vert A\vert^2-(c-1)\varepsilon\vert A\vert^2(1-\langle X, \nu\rangle)+2\varepsilon H\\
\ge{}&\varepsilon(2-(2c+1)H)H\\
>{}&0
\end{split}
\]
which is absurd. Now take $\e\to 0$ to obtain $\inner{\cd H}{X}\ge 0$ in $M_t$ for $t\in (-\infty,T]$. 
Applying the maximum principle as in the proof of Lemma \ref{lem:old-but-not-ancient}, implies that $\inner{\cd H}{X}$ remains non-negative up to time $0$.

Having established the first claim, the second follows as in Lemma~\ref{lem:crude estimates}.
\end{proof}

\begin{proposition}\label{prop:linear analysis}
If we define $A\in(0,\infty)$ as in \eqref{eq:backwards asymptotics}, then
\[
\mathrm{e}^{-\lambda_0^2t}y(x,t)\to A\phi_0(\lambda_0\vert x\vert)\;\;\text{uniformly as}\;\; t\to-\infty\,.
\]
\end{proposition}
\begin{proof}
Given $\tau<0$, consider the rescaled height function
\[
y^{\tau}(x,t)\doteqdot \mathrm{e}^{-\lambda_0^2\tau}y(x,t+\tau)\,,
\]
which is defined on the time-translated flow $\{M^\tau_t\}_{t\in(-\infty,-\tau)}$, where $M^\tau_t\doteqdot M_{t+\tau}$. Note that
\begin{equation}\label{eq:heat equation on the flow}
\left\{\begin{aligned}
(\pd_t-\Delta^\tau)y^{\tau}={}&0\;\;\text{in}\;\;\{M_t^\tau\}_{t\in(-\infty,-\tau)}\\
\inner{\cd^\tau y^\tau}{N}={}&y\;\;\text{on}\;\;\{\pd M_t^\tau\}_{t\in(-\infty,-\tau)}\,,
\end{aligned}\right.
\end{equation}
where $\cd^\tau$ and $\Delta^\tau$ are the gradient and Laplacian on $\{M^\tau_t\}_{t\in(-\infty,-\tau)}$, respectively, and $N$ is the outward unit normal to $\pd B^{n+1}$. 

Since $\{M_t\}_{t\in(-\infty,0)}$ reaches the origin at time zero, it must intersect the constructed solution for all $t<0$. In particular, the value of $\underline y$ on the former can at no time exceed the value of $\overline y$ on the latter. But then \eqref{eq:backwards asymptotics} and \eqref{eq:min y lower bound general} yield
\begin{equation}\label{eq:sup bound sublimit}
\limsup_{t\to-\infty}\mathrm{e}^{-\lambda_0^2t}\overline y<\infty\,.
\end{equation}
This implies a uniform bound for $y^\tau$ on $\{M^\tau_t\}_{t\in(-\infty,T]}$ for any $T\in\R$. So Alaoglu's theorem yields a sequence of times $\tau_j\to-\infty$ such that $y^{\tau_j}$ converges in the weak$^\ast$ topology as $j\to\infty$ to some $y^\infty\in L^2_{\mathrm{loc}}(\overline B{}^n\times (-\infty,\infty))$. Since convexity and the boundary condition imply a uniform bound for $\cd^\tau y^\tau$ on any time interval of the form $(-\infty,T]$, we may also arrange that the convergence is uniform in space at time zero, say.

We conclude from Corollary \ref{cor:Liouville} that
\[
y^\infty(x,t)=A\mathrm{e}^{-\lambda_0^2t}\phi_0(\lambda_0\vert x\vert)
\]
for some $A\ge 0$. In particular,
\[
\mathrm{e}^{-\lambda_0^2\tau_j}y(x,\tau_j)=y^{\tau_j}(x,0)\to A\phi_0(\lambda_0\vert x\vert)\;\;\text{uniformly as}\;\; j\to\infty\,.
\]
Now, if $A$ is not equal to the corresponding value on the constructed solution (note that the full limit exists for the latter), then one of the two solutions must lie above the other at time $\tau_j$ for $j$ sufficiently large. But this violates the avoidance principle.
\end{proof}

\subsection{Uniqueness}

Uniqueness of the constructed ancient solution now follows directly from the avoidance principle (cf. \cite{FBCSF}).

\begin{proposition}\label{prop:uniqueness}
Modulo time translation and rotation about the origin, there is only one convex, locally uniformly convex ancient solution to free boundary mean curvature flow in the ball.
\end{proposition}
\begin{proof}
Denote by $\{M_t\}_{t\in(-\infty,0)}$ the constructed ancient solution and let $\{M'_t\}_{t\in(-\infty,0)}$ be a second ancient solution which, without loss of generality, contracts to the point $e_{n+1}$ at time $0$. Given any $\t>0$, consider the time-translated solution $\{M_t^\tau\}_{t\in(-\infty,-\tau)}$ defined by $M_t^\t=M'_{t+\tau}$. By Proposition \ref{prop:linear analysis},
\[
\mathrm{e}^{-\lambda_0^2t}y^\tau(x,t)
\to A\mathrm{e}^{\lambda_0^2\tau}\phi_0(\lambda_0\vert x\vert)\;\;\text{as}\;\; t\to-\infty
\]
uniformly in $x$. So $M^\tau_t$ lies above $M_t$ for $-t$ sufficiently large. The avoidance principle then ensures that $M^\tau_t$ lies above $M_t$ for all $t\in(-\infty,0)$. Taking $\tau\to 0$, we find that $M'_t$ lies above $M_t$ for all $t<0$. Since the two solutions both reach the point $e_{n+1}$ at time zero, they intersect for all $t<0$ by the avoidance principle. The strong maximum principle then implies that the two solutions coincide for all $t$.
\end{proof}

\bibliographystyle{acm}
\bibliography{../../../bibliography}

\end{document}